\documentclass[11pt]{article}
\usepackage[margin=1in]{geometry}
\usepackage{graphicx}
\usepackage{enumerate}
\usepackage{amsmath}
\usepackage{amssymb}
\usepackage[colorlinks=true,allcolors=black]{hyperref}
\usepackage{xcolor}
\usepackage{amsthm}
\usepackage[round]{natbib}
\DeclareMathOperator*{\argmin}{arg\,min}
\allowdisplaybreaks
\newtheorem{theorem}{Theorem}
\newtheorem{proposition}[theorem]{Proposition}
\newtheorem{lemma}[theorem]{Lemma}
\newtheorem{definition}[theorem]{Definition}
\newtheorem{assumption}[theorem]{Assumption}

\newcommand{\tr}{\mathrm{Tr}}
\newcommand{\cL}{\mathcal L}
\newcommand{\cA}{\mathcal A}
\newcommand{\cP}{\mathcal P}
\newcommand{\cN}{\mathcal N}
\newcommand{\cF}{\mathcal F}
\newcommand{\cE}{\mathcal E}
\newcommand{\cR}{\mathcal R}
\newcommand{\bH}{\mathbf{H}}
\newcommand{\bV}{\mathbf{V}}
\newcommand{\R}{\mathbb R}
\newcommand{\E}{\mathbb E}
\newcommand{\Fisher}{\mathcal I}
\newcounter{shpialgorithm}

\begin{document}
\title{Amortized Score–Hamiltonian Policy Iteration: A Grid-Free Scheme for Relaxed Stochastic Control Problems}
\author{
Qi Feng\thanks{\noindent Department of
Mathematics, Florida State University, Tallahassee, 32306; email: qfeng2@fsu.edu. This author is partially supported by the National Science Foundation under grant \#DMS-2420029.}
~ and ~ Gu Wang\thanks{ \noindent Department of
Mathematical Sciences, Worcester Polytechnic Institute, Worcester, MA 01609;
email: gwang2@wpi.edu. This author is partially supported
by National Science Foundation grant \#DMS - 2206282.}
}
\date{}
\maketitle
\begin{abstract}
We develop an amortized, grid-free implementation of continuous Langevin dynamics based policy-value iteration for entropy-regularized, infinite-horizon relaxed stochastic control problems. The improvement rate of the exact iteration is a discounted aggregate of relative Fisher information between the policy and the Gibbs law of its Hamiltonian. The associated score residual is the velocity with which the control's Langevin dynamics transport its law. 
 We project this velocity onto a conditional sampler
shared across states, instead of one Langevin dynamics per state, and the value dynamics
onto a parametric critic, estimating both projections at sampled states to obtain coupled actor--critic flows.
The score loss measures the actor's agreement with the current critic, while the policy-evaluation residual measures the critic's agreement with the actor.  We also derive gradient and Hessian residuals, including a Feynman--Kac representation for the gradient equation, to control errors not detected by the projected value iteration. An exact decomposition of the HJB residual combines these errors into a policy-suboptimality bound under verification and logarithmic Sobolev assumptions. In the linear-quadratic class, both projections are exact and recover the pointwise iteration, and we provide numerical experiments on general models to demonstrate the coupled actor--critic learning in high-dimensions.
\end{abstract}
{$\mathbf{Keywords}$.} Hamilton-Jacobi-Bellman equation; Reinforcement learning; Policy iteration; Score matching; Diffusion policy; Relaxed control; Langevin dynamics.


\section{Introduction}\label{sec: intro}
Policy iteration algorithms, in which the policy of a stochastic control problem and its value are improved alternately, have recently received renewed attention in continuous-time models, in particular within the relaxed control framework \citep{fleming1984relaxed, zhou1992relaxed} with entropy regularization, where the randomization of the control is interpreted as exploration in the sense of reinforcement learning \citep{wang2020reinforcement, tang2022exploratory, reisinger2021regularity}. Convergence of policy iteration for entropy-regularized problems is established in \citet{huang2025convergence, ma2026convergence, tran2025policy}, and for classical controlled diffusions in \citet{kerimkulov2020exponential, kerimkulov2021modified}. In parallel, \citet{jia2022policy_evaluation, jia2022policy_gradient, jia2023q} developed the martingale approach to policy evaluation, policy gradient and $q$-learning in continuous time, where the optimal policy is the Gibbs measure generated from the $q$-function, and \citet{zhao2023policy} formulated policy optimization through occupation times. Convergence rates of policy gradient methods were obtained in \citet{reisinger2022linear, giegrich2022convergence}, with regret analysis in \citet{tang2024regret}. More recent contributions include the random-measure formulation of continuous-time reinforcement learning of \citet{bender2026continuous}, fast policy learning for entropy-regularized linear-quadratic control \citep{guo2026fast}, and the accuracy of discretely sampled stochastic policies \citep{jia2026accuracy}. In conventional policy iteration, the state evolves in continuous time while policies are updated in discrete steps, each using an exact or approximate policy evaluation, obtained from a partial differential equation (PDE) or a learning algorithm.

A second line of work makes the iteration itself continuous. \citet{siska2024gradient} construct gradient flows on the space of probability measures for regularized control problems, \citet{kerimkulov2023fisher, kerimkulov2024mirror, sethi2024entropy} study Fisher-Rao gradient flows, mirror descent and entropy annealing for measure-valued controls, \citet{zhang2018policy} treat policy optimization as a Wasserstein gradient flow, and \citet{zhou2023policy, zhou2024solving} propose actor-critic flows with linear convergence. In \citet{feng2026continuous}, we introduced a continuous policy-value iteration in which the approximations of the value function and of the control are updated simultaneously through coupled Langevin-type dynamics in a fictitious time $\tau$, for relaxed control problems with controlled drift and diffusion. Policy improvement holds along the flow, and convergence is established under the additional monotonicity and dissipativity hypotheses. 
The exact iteration is pointwise: at every state $x$ the policy $\pi(x,\cdot)$ follows its
own Langevin dynamics in the control variable, and the value is evolved as a function of
the state. Both are out of reach in high dimension. The goal of this paper is to develop a numerical method which implement the this iteration scheme in high dimensions for both control and state. 

The key observation is contained in the proof of the policy improvement result of \citet{feng2026continuous}: the rate of improvement is a discounted integral, along the state trajectory and over the policy history, of $\|\nabla_u\bH-\lambda\nabla_u\ln\pi\|^2$ averaged under the corresponding policy $\pi$, where $\bH$ is the Hamiltonian and $\lambda$ the entropy parameter. Since the Gibbs policy of the Hamiltonian has score $\nabla_u\bH/\lambda$, this quantity is $\lambda^2$ times the relative Fisher information between the current policy and the current Gibbs target. Fisher-type discrepancies also underlie score matching \citep{hyvarinen2005estimation, vincent2011connection}, which is the training principle behind score-based diffusion models \citep{song2019generative, song2021score}. In discrete-time reinforcement learning, diffusion models have recently been used as policy classes \citep{wang2023diffusion, hansen2023idql, chi2023diffusion, ren2024diffusion, wang2024diffusion, zhu2023diffusion}, going back to the energy-based policies of \citet{haarnoja2017reinforcement}, and \citet{psenka2024learning} train a diffusion policy by matching its score to the action-gradient of a $Q$-function. Our identity relates this principle to continuous-time control, with the Hamiltonian in place of the $Q$-function: a discounted integral of the score discrepancy over the policy history is the improvement rate of the iteration of \citet{feng2026continuous}, and hence the dissipation rate of its value gap $V-v^\tau$. 

It also shows that the residual of the match $\nabla_u\bH - \lambda\nabla_u\ln\pi$ is the velocity field with which the Langevin dynamics of \citet{feng2026continuous} transport the law of the control. What we propose in this paper is to move the samples of the policy along this velocity, as far as the policy class allows, instead of to descend the loss. The main difference between our method and the continuous-time form of $Q$-score matching, as proposed by \citet{hua2025continuous}, is that: in the latter the action itself is a diffusion in physical time whose drift is the control. The drift is penalized by a quadratic cost, and the ``score'' matched to $\lambda^{-1}\nabla_aQ$ is the first-order condition of this cost. The Boltzmann law of $Q$ appears as the stationary law of the action at a frozen state, and $Q$ is learned from data through a martingale condition. In our method, we assume that the model is known, and the policy is a conditional density on the control space. The Gibbs law of the Hamiltonian is the exact maximizer of the relaxed HJB equation, and the discounted score discrepancy gives the improvement rate of an iteration in fictitious time.

Based on the above observation, we develop a score-Hamiltonian policy iteration for relaxed control problems in which both the policy and value parts of the pointwise iteration are projected onto parametric families and run in its fictitious time. The value function (the critic) is a parametric family updated by the projection of the value dynamics onto the parameters, a Dirac--Frenkel projection whose relation to the pointwise flows is made explicit in Section \ref{sec:parametric-flows}.  Then we amortize the state-wise dynamics for the control, in
the sense of amortized inference \citep{gershman2014amortized,kingma2014auto}. A single conditional sampler (the actor) $u=T_\phi(x,\xi)$, shared by
all states, is moved by projecting the Langevin velocities $\nabla_u\bH - \lambda\nabla_u\ln\pi_\phi$ of \eqref{eq:u-dynamics-r} at sampled states onto its
parameters, so a policy sample at any state costs one evaluation of $T_\phi$.  Thus no grid in the state space is needed. 

We show that the HJB residual of the critic decomposes into the policy-evaluation residual plus $\lambda$ times the Kullback-Leibler divergence of the actor from the Gibbs law of the critic, the latter bounded by the score discrepancy when the Gibbs law satisfies a logarithmic Sobolev inequality. Together with verification and occupation-measure assumptions, these residuals yield a bound on the suboptimality of the learned policy. The joint zero-residual solutions for policy and value iterations satisfy the Hamilton-Jacobi-Bellman (HJB) equation and its Gibbs policy relation and thus uncover the optimal solution to the stochastic control problems. In the linear-quadratic case, our policy and value iterates are obtained in closed form, where the projected Langevin flow is the moment equation of the Langevin iterate, and the projected value dynamics coincide with a Lyapunov flow and the stationary points satisfy an algebraic Riccati equation. 

This paper is organized as follows. In Section \ref{sec: setting} we recall the relaxed control problem and the continuous policy-value iteration of \citet{feng2026continuous}. We represent the improvement rate using the relative Fisher information. Section \ref{sec: relaxed} introduces the parametric actor and critic flows, presents their numerical algorithm, and describes the residual diagnostics and derivative calculations used in the experiments. Section \ref{sec: LQ} discusses the linear-quadratic case, and Section \ref{sec: numerics} report numerical experiments including bimodal actor in high-dimensions. 

\section{Stochastic control problems and the continuous policy-value iteration}\label{sec: setting}
In this section we briefly review the continuous policy-value iteration of \citet{feng2026continuous} for the relaxed control problem, and provide a new characterization of its improvement rate in terms of relative Fisher information. Throughout, $W$ is a $d$-dimensional Brownian motion on a filtered probability space $(\Omega,\cF,\{\cF^W_t\}_{t\geq 0},\mathbb P)$, the state variable $X$ takes values in $\R^d$, the control takes values in $U = \R^n$, and $b:\R^d\times U\to\R^d$, $\sigma:\R^d\times U\to\R^{d\times k}$, $f:\R^d\times U\to\R$ are time-homogeneous. The discount rate is a constant $\beta>0$. Define
\begin{equation}\label{notation: hamiltonian}
\bH(x,u,p,P) := f(x,u) + \frac{1}{2}\tr[\sigma\sigma^\top(x,u)P] + b^\top(x,u)p,\qquad (x,u,p,P)\in\R^d\times U\times\R^d\times\mathbb S^d,
\end{equation}
as the Hamiltonian. For matrices we use the Frobenius contraction $A:B=\tr(A^\top B)$. We write $\bH(x,u,v_x,v_{xx})$ for $\bH(x,u,v_x(x),v_{xx}(x))$ if $v$ is a function of $x$.

\subsection{Relaxed control problem}\label{sec: relaxed setting}
In the relaxed control framework the agent chooses, at every time $t\geq 0$ and state $x$, a probability density $\pi_t(x,\cdot)$ on $U$ rather than a control value, and the state follows the exploratory dynamics \citep{wang2020reinforcement}
\begin{equation}\label{eq:exploratory}
\begin{gathered}
dX_t = \bar b(X_t,\pi_t)dt + \bar\sigma(X_t,\pi_t)dW_t,\quad X_0 = x,\\
\bar b(x,\pi) = \int_U b(x,u)\pi(x,u)du,\qquad \bar\sigma(x,\pi) = \left(\int_U\sigma\sigma^\top(x,u)\pi(x,u)du\right)^{1/2}.
\end{gathered}
\end{equation}
Here the square root is the symmetric nonnegative $d\times d$ square root. The agent chooses the control $\pi = \{\pi_t\}_{t\geq 0}$ from the admissible set $\cA(x)$, to maximize the cumulative discounted payoff, regularized by the differential entropy of the policy with a parameter $\lambda>0$, i.e., $V(x) = \sup_{\pi\in\cA(x)}J(x,\pi)$, where
\begin{equation}\label{eq:J-relaxed}
J(x,\pi) = \E\left[\int_0^\infty e^{-\beta t}\left(\int_U f(X_t,u)\pi_t(X_t,u)du - \lambda\int_U\pi_t(X_t,u)\ln\pi_t(X_t,u)du\right)dt\right].
\end{equation}
\begin{definition}\label{def:admissible-relaxed}
Let $\cP(U)$ denotes the set of probability densities on $U$. For any initial state $x\in\R^d$, the admissible set $\cA(x)$ is the collection of policies $\pi = \{\pi_t(y,\cdot)\in\cP(U),\ t\geq 0,\ y\in\R^d\}$ such that $(t,y,u)\mapsto\pi_t(y,u)$ is measurable, $\pi_t(y,\cdot)$ is a probability density on $U$ for every $(t,y)$, the coefficient integrals in \eqref{eq:exploratory} and the reward and entropy integrals are finite, the exploratory dynamics admit a unique nonexplosive strong solution, and
\begin{enumerate}
\item[(i)] $\E\left[\int_0^T\left(|\bar b(X_t,\pi_t)|^2 + \tr[\bar\sigma\bar\sigma^\top(X_t,\pi_t)]\right)dt\right]<\infty$ for each $T>0$;
\item[(ii)] $\E\left[\int_0^\infty e^{-\beta t}\left|\int_U f(X_t,u)\pi_t(X_t,u)du - \lambda\int_U\pi_t(X_t,u)\ln\pi_t(X_t,u)du\right|dt\right]<\infty$.
\end{enumerate}
A policy is stationary if $\pi_t = \bar\pi:\R^d\to\cP(U)$ does not depend on $t$, which we refer to as a policy field, and we write $J(x,\pi)$ as $J(x,\bar\pi)$.
\end{definition}
For a stationary policy, write
\begin{equation}\label{eq:reward-generator}
\begin{aligned}
r_\pi(x)&=\int_U\big(f(x,u)-\lambda\ln\pi(x,u)\big)\pi(x,u)du,\\
\cL^\pi h&=\bar b(\cdot,\pi)\cdot\nabla_xh+
\tfrac12\tr[\bar\sigma\bar\sigma^\top(\cdot,\pi)\nabla_x^2h].
\end{aligned}
\end{equation}
Its classical policy-evaluation equation is
\begin{equation}\label{eq:policy-evaluation}
\beta J(\cdot,\pi)=r_\pi+\cL^\pi J(\cdot,\pi),
\end{equation}
under the regularity conditions laid out below.

\begin{assumption}\label{assumption:parameter}

\begin{enumerate}
\item[(i)] $f$, $b$ and $\sigma$ are $C^2$ in $(x,u)$; $b$ and $\sigma$ have at most linear growth in $(x,u)$, and $f$ together with the derivatives of $f$, $b$, $\sigma$ up to order two have at most polynomial growth in $(x,u)$.
\item[(ii)] $\sigma\sigma^\top$ is uniformly positive definite, i.e.\ $\sigma\sigma^\top(x,u)\succeq c_0\,I_d$ for some constant $c_0>0$;
\item[(iii)] For every value function $v$ considered below --- the iterates $v^\tau$ of Section \ref{sec: FW} and the critics $v_\theta$ of Section \ref{sec: relaxed} --- there are constants $C_v, a_v>0$, $r_v>0$ and $m_v\geq0$ with
\begin{equation}\label{eq:confinement}
\bH\big(x,u,v_x(x),v_{xx}(x)\big)\leq C_v\big(1 + |x|^{m_v}\big) - a_v|u|^{r_v}\qquad\text{for all }(x,u)\in\R^d\times U;
\end{equation}
\item[(iv)] $V(x)\in\R$ for every $x\in\R^d$;
\item[(v)] For every policy density considered below --- $\pi_\tau(x,\cdot)$ of Section \ref{sec: FW} and  $\pi_\phi(x,\cdot)$ \ref{sec: relaxed} --- is strictly positive and $C^1$ on $U$, with finite entropy, $\nabla_u\ln\pi(x,\cdot)\in L^2(\pi(x,\cdot))$, and $\int_U(|\bH(x,u,v_x,v_{xx})|+|\nabla_u\bH(x,u,v_x,v_{xx})|^2)\pi(x,u)\,du<\infty$ for every value function $v$ against which it is evaluated.
\end{enumerate}
\end{assumption}

The HJB equation associated with the relaxed control problem is
\begin{equation}\label{eq:HJB-relaxed}
\begin{aligned}
&-\beta v(x) + \sup_{\pi\in\cP(U)}\int_U\Big(\bH(x,u,v_x,v_{xx}) - \lambda\ln\pi(u)\Big)\pi(u)du\\
=&\ -\beta v(x) + \lambda\ln\int_U\exp\left(\frac{\bH(x,u,v_x,v_{xx})}{\lambda}\right)du = 0.
\end{aligned}
\end{equation}
The supremum over densities is attained, whenever the integral is finite, at the Gibbs density of the Hamiltonian,
\begin{equation}\label{eq:gibbs}
\Gamma[v](x,u) := \frac{\exp\left(\frac{1}{\lambda}\bH(x,u,v_x(x),v_{xx}(x))\right)}{\int_U\exp\left(\frac{1}{\lambda}\bH(x,u',v_x(x),v_{xx}(x))\right)du'}.
\end{equation}
If $\hat v=V$ is a classical solution of \eqref{eq:HJB-relaxed}, its Gibbs policy $\hat\pi=\Gamma[\hat v]$ is optimal  with standard verification argument. 

\subsection{The continuous policy-value iteration}\label{sec: FW}
As observed in \citet{feng2026continuous}, for a given $x$ and an $n$-dimensional Brownian motion $B$ independent of $W$, the optimal $\hat\pi(x,\cdot)$ is the invariant distribution of the Langevin diffusion
\begin{equation}\label{eq:langevin-optimal}
du^x_\tau = \nabla_u\bH(x,u^x_\tau,\hat v_x,\hat v_{xx})d\tau + \sqrt{2\lambda}dB_\tau,
\end{equation}
if this diffusion is well posed and nonexplosive and $\int_U\exp(\bH(x,u,\hat v_x,\hat v_{xx})/\lambda)du<\infty$ \citep{geman1986diffusions, roberts1996exponential}.

Motivated by this, the iteration scheme in \citet{feng2026continuous} updates a policy and its value in a fictitious time $\tau\geq 0$. For each state $x\in\R^d$, given an initial distribution $\pi_0(x,\cdot)$ and $u^x_0\sim\pi_0(x,\cdot)$, the policy iteration is the Langevin-type dynamics
\begin{equation}\label{eq:u-dynamics-r}
du^x_\tau = \nabla_u\bH\big(x,u^x_\tau,v^\tau_x(0,x),v^\tau_{xx}(0,x)\big)d\tau + \sqrt{2\lambda}dB_\tau,\qquad \pi_\tau(x,\cdot) := \text{the law of }u^x_\tau,
\end{equation}
$v^\tau$ is the iterate for the value function which is updated simultaneously with $u_\tau$. It rests on the following construction, which uses the infinite horizon in an essential way: the most recent iterate is applied at $t=0$, and the older iterates are pushed into the discounted future.
\begin{definition}[\citealp{feng2026continuous}, Definition 3.1]\label{def:shift}
For each $\tau\geq 0$, given iterates $\{\pi_s(x,\cdot),\ 0\leq s\leq\tau\}$ for every $x$, define the time-dependent policy
\begin{equation}\label{eq:pi-reverse}
\pi^\tau_t(x,\cdot) = \begin{cases}\pi_{\tau-t}(x,\cdot), & 0\leq t\leq\tau;\\ \pi_0(x,\cdot), & t>\tau.\end{cases}
\end{equation}
Let $X^\tau$ be the solution to \eqref{eq:exploratory} under $\pi^\tau$, and $v^\tau(t,x)$ be its corresponding value with initial value $X^\tau_t=x$. The value iterate is $v^\tau(x):=v^\tau(0,x)$, and in particular $v^0=J(\cdot,\pi_0)$.
\end{definition}
Since the coefficients of the state dynamics are time-homogeneous, it can be verified that
\begin{equation}\label{eq:shift-identity}
v^\tau(t,x) = v^{\tau-t}(0,x),\quad 0\leq t\leq\tau,\qquad\text{hence}\qquad \frac{d}{d\tau}v^\tau(0,x) = -\partial_t v^\tau(t,x)\Big|_{t=0},
\end{equation}
which is what allows the family $\{v^\tau\}$ to be flowed in $\tau$. Denote as $\bH^\tau := \bH(x,u,v^\tau_x(x),v^\tau_{xx}(x))$ and $g^\tau(x) := \int_U(\bH^\tau - \lambda\ln\pi_\tau)\pi_\tau(x,u)du$, the value iteration dynamics are
\begin{equation}\label{eq:v-dynamics-r}
\begin{aligned}
dv^\tau(x) &= \left(-\beta v^\tau(x) + g^\tau(x)\right)d\tau;\\
dv^\tau_x(x) &= \left(-\beta v^\tau_x(x) + \nabla_x g^\tau(x)\right)d\tau;\\
dv^\tau_{xx}(x) &= \left(-\beta v^\tau_{xx}(x) + \nabla^2_{xx}g^\tau(x)\right)d\tau,
\end{aligned}
\end{equation}
where the $x$-derivatives act on the policy field $\pi_\tau(\cdot,\cdot)$ as well. \citet{feng2026continuous} establishes exponential convergence to $(\Gamma[V],V)$ under the additional monotonicity and dissipativity hypotheses of its convergence analysis.

The iteration is indexed pointwise by $x$. A direct grid implementation stores the policy, value, and required derivatives at every grid point, which becomes expensive as the state dimension grows. The main goal of this paper is to replace this representation by parametric families and sampled states, which enables iterations in high dimensions.

\subsection{Relative Fisher information and {the transport of the policy law}}\label{sec: fisher}
In this section we discuss the improvement of the above policy-value iteration in $\tau$ direction, and we make the following regularity assumptions.
\begin{assumption}\label{assumption-T}
The pointwise coupled flow for policy and value exists on each finite iteration interval $[0,T]$. Furthermore,
\begin{enumerate}
\item[(i)] The functions $v^\tau(x)$ are twice continuously differentiable in $\tau$ and four times in $x$, with continuous mixed derivatives.
\item[(ii)] With $a^\tau(x,u):=\nabla_u\bH(x,u,v_x^\tau,v_{xx}^\tau)-\lambda\nabla_u\ln\pi_\tau(x,u)$,
\[
\lim_{R\to\infty}\int_{\partial B_R}
 (\bH^\tau-\lambda\ln\pi_\tau)\pi_\tau a^\tau\cdot n\,dS=0,
\]
locally uniformly in $(\tau,x)$, where $B_R\subset\R^n$ and $n$ is its outward unit normal.
\item[(iii)] Writing $w^\tau=\partial_\tau v^\tau$, for every $T>0$ and $x$,
\[
\begin{aligned}
&\E_x\int_0^T e^{-\beta s}\int_U|a^{T-s}(X_s^T,u)|^2\pi_{T-s}(X_s^T,u)du\,ds<\infty,\\
&\E_x\int_0^T e^{-2\beta s}\big|\bar\sigma^\top(X_s^T,\pi_{T-s})\nabla_x w^{T-s}(X_s^T)\big|^2ds<\infty.
\end{aligned}
\]
\end{enumerate}
\end{assumption}

For probability densities $\pi,\gamma$ on $\R^n$ with $\nabla_u\ln(\pi/\gamma)\in L^2(\pi)$, the relative Fisher information of $\pi$ with respect to $\gamma$ is
\begin{equation}\label{eq:fisher-def}
\Fisher(\pi\,\|\,\gamma) := \int_{\R^n}\left\|\nabla_u\ln\frac{\pi(u)}{\gamma(u)}\right\|^2\pi(u)du \geq 0,
\end{equation}
with $\Fisher(\pi\|\gamma) = 0$ if and only if $\pi = \gamma$ when both are strictly positive and $C^1$ on the connected set $U = \R^n$, which Assumption \ref{assumption:parameter}(v) grants. The following lemma rewrites the improvement rate of the continuous iteration above in terms of relative Fisher information.
\begin{lemma}\label{lemma:fisher}
Under Assumptions \ref{assumption:parameter} and \ref{assumption-T}, for every $\tau>0$ and $x\in\R^d$, the policy and value iteration \eqref{eq:u-dynamics-r} and \eqref{eq:v-dynamics-r} satisfy
\begin{equation}\label{eq:improvement-rate}
\frac{d}{d\tau}v^\tau(x) = \lambda^2\,\E\left[\int_0^\tau e^{-\beta s}\,\Fisher\Big(\pi_{\tau-s}(X^\tau_s,\cdot)\,\Big\|\,\Gamma[v^{\tau-s}](X^\tau_s,\cdot)\Big)ds\,\Big|\,X^\tau_0 = x\right]\geq 0.
\end{equation}
\end{lemma}
\begin{proof}
The Gibbs normalizer is independent of $u$, so
\begin{equation}\label{eq:target-score}
\nabla_u\ln\Gamma[v](x,u)=\lambda^{-1}\nabla_u\bH(x,u,v_x,v_{xx}),
\qquad
\nabla_u\bH-\lambda\nabla_u\ln\pi=-\lambda\nabla_u\ln\frac{\pi}{\Gamma[v]}.
\end{equation}
Consequently,
\begin{equation}\label{eq:fisher-identity}
\int_U|\nabla_u\bH-\lambda\nabla_u\ln\pi|^2\pi\,du
=\lambda^2\Fisher(\pi\,\|\,\Gamma[v]).
\end{equation}
Differentiate with resepct to $\tau$ in on both sides of the first equation of \eqref{eq:v-dynamics-r} and set $w^\tau:=\partial_\tau v^\tau$. The terms obtained by differentiating $v^\tau$ inside the Hamiltonian yields $\cL^{\pi_\tau}w^\tau$. For the remaining terms (involving $\pi_\tau$), by the Fokker--Planck equation $\partial_\tau\pi_\tau=-\nabla_u\cdot(\pi_\tau a^\tau)$
\begin{align}
\partial_\tau\int_U(\bH^\tau-\lambda\ln\pi_\tau)\pi_\tau\,du = -\lambda\int_U\partial_\tau\pi_\tau\,du + \int_U(\bH^\tau-\lambda\ln\pi_\tau)\partial_\tau\pi_\tau\,du
=\int_U|a^\tau|^2\pi_\tau\,du.
\end{align}
where the first term vanishes because $\pi_\tau$ vanishes at the boundary of $U$ and the second term is calculated by integration by parts. Thus
\[
\partial_\tau w^\tau=(\cL^{\pi_\tau}-\beta)w^\tau
+\lambda^2\Fisher(\pi_\tau\,\|\,\Gamma[v^\tau]),
\qquad w^0=0,
\]
where the initial condition follows from the policy-evaluation equation for $v^0=J(\cdot,\pi_0)$. Apply It\^o's formula to $e^{-\beta s}w^{\tau-s}(X_s^\tau)$ on $0\leq s\leq\tau$, and Assumption \ref{assumption-T}(iii) makes its stochastic integral a true martingale, and its terminal value is zero. Taking expectations implies \eqref{eq:improvement-rate}.
\end{proof}

Fix a state $x$ and write $\bH^\tau(x,u)=\bH(x,u,v_x^\tau,v_{xx}^\tau)$. The law $\pi_\tau(x,\cdot)$ of the Langevin iterate \eqref{eq:u-dynamics-r} solves the Fokker--Planck equation, which is a continuity equation:
\begin{equation}\label{eq:continuity}
\partial_\tau\pi_\tau
=-\nabla_u\cdot(\pi_\tau\nabla_u\bH^\tau)+\lambda\Delta_u\pi_\tau
=-\nabla_u\cdot(\pi_\tau\bV_\tau),
\end{equation}
with the velocity
\begin{equation}\label{eq:langevin-velocity}
\bV_\tau(x,u)
:=\nabla_u\bH^\tau(x,u)-\lambda\nabla_u\ln\pi_\tau(x,u)
=-\lambda\nabla_u\ln\frac{\pi_\tau(x,u)}{\Gamma[v^\tau](x,u)}.
\end{equation}
One important observation is that the iteration uses the Langevin dynamics only through this law: the policy at $\tau$ is $\pi_\tau$. A family of deterministic sampling maps $u=T(x,\xi)$ cannot reproduce the Brownian paths of \eqref{eq:u-dynamics-r}, but it can follow the law, by moving its samples with $\bV_\tau$. If the critic is instead frozen at $v$,
\begin{equation}\label{eq:kl-dissipation}
\frac{d}{d\tau}\bigg[\lambda\,\mathrm{KL}\big(\pi_\tau(x,\cdot)\,\|\,\Gamma[v](x,\cdot)\big)\bigg]
=-\int_U\|\bV_\tau(x,u)\|^2\pi_\tau(x,u)\,du,
\end{equation}
where {\eqref{eq:langevin-velocity}} is evaluated with $v^\tau=v$. {Thus the score residual is the velocity of the law, and its mean square is the frozen-target free-energy dissipation rate. The iteration scheme developed in Section \ref{sec:parametric-flows} moves the samples of a parametric sampler with the projection of this velocity onto the motions that the sampler can realize.}

\section{Amortized Score-Hamiltonian policy-value iteration}\label{sec: relaxed}
The main results of this paper is the numerical method which realize the continuous policy-value iteration in \cite{feng2026continuous} in high dimensions. It represents the policy by a conditional sampler and the value by a smooth function. Their parameters evolve together, using batches of states in place of a state grid. We first give the two parameter flows and their implementation, then the residuals used to monitor the computation.

\subsection{Parametric actor and critic flows}\label{sec:parametric-flows}
Let $\xi$ have a fixed reference law and define the following parametric families
\begin{equation}\label{eq:parametric-pair}
u=T_\phi(x,\xi)\sim\pi_\phi(x,\cdot),\qquad
v=v_\theta(x),\qquad \phi\in\R^{m_a},\quad\theta\in\R^{m_c}.
\end{equation}
The actor $T_\phi$ generates controls, and the critic $v_\theta$ supplies their Hamiltonian target through its spatial derivatives. In the implementation below, the actor has an evaluable log-density and score $s_\phi=\nabla_u\ln\pi_\phi$, e.g. Gaussian samplers and invertible normalizing flows. Assume that the parametric family so that $T_\phi(x,\xi)$ has positive smooth densities, which are differentiable in parameterizations, and are integrable as needed for the expectations and derivatives below. The sampling law $\rho$ weights the states on which the approximation is trained.

\paragraph{Actor flow.}\label{sec: score actor}
Given the current critic $v_\theta$, write $\bH_\theta(x,u)=\bH(x,u,v_{\theta,x},v_{\theta,xx})$. The Langevin iterate transports its law with the velocity \eqref{eq:langevin-velocity}, which equals
\begin{equation}\label{eq:actor-velocity}
\bV_{\theta,\phi}(x,u)=\nabla_u\bH_\theta(x,u)-\lambda s_\phi(x,u),
\end{equation}
at the current actor and critic.
A parameter velocity $h$ moves an actor sample with velocity $\partial_\phi T_\phi h$. We choose the closest such motion to \eqref{eq:actor-velocity}:
\begin{equation}\label{eq:projected-langevin}
\dot\phi=\argmin_h\int_{\R^d}\E_\omega
\big\|\partial_\phi T_\phi(x,\xi)h-\bV_{\theta,\phi}(x,T_\phi(x,\xi))\big\|^2\rho(dx).
\end{equation}
$\E_\omega$ is the expectation over the reference noise $\xi\sim\omega$ at fixed $x$, $\E_\omega[F(x,\xi)]=\int F(x,\xi)\,\omega(d\xi)$. Fixing $x$ and a noise label $\xi$, if the parameters move with velocity $h\in\R^{m_a}$, the sample $T_\phi(x,\xi)$ moves deterministically, at the velocity
\[
\frac{d}{d\epsilon}T_{\phi+\epsilon h}(x,\xi)\Big|_{\epsilon=0}=\partial_\phi T_\phi(x,\xi)\,h\in\R^n,\qquad \partial_\phi T_\phi(x,\xi)\in\R^{n\times m_a},
\]
while the Langevin iterate asks the sample at $u$ to move with $\bV_{\theta,\phi}(x,u)=\nabla_u\bH_\theta(x,u)-\lambda s_\phi(x,u)$. The difference $\partial_\phi T_\phi(x,\xi)h-\bV_{\theta,\phi}(x,T_\phi(x,\xi))$ is therefore the velocity error of one sample, and $\E_\omega$ averages its squared Euclidean norm over all samples at $x$, that is, over the controls of $\pi_\phi(x,\cdot)$. Integrating against $\rho$ and minimizing over $h$ selects the parameter velocity whose induced motion of the samples, at all sampled states, is closest in mean square to the Langevin velocity, and matching the velocities makes the laws evolve alike.

This is a linear least-squares problem, with Gram matrix and right-hand side rewritten as, respectively
\begin{equation}\label{eq:langevin-metric}
\begin{aligned}
G_T(\phi)&=\int\E_\omega[\partial_\phi T_\phi^\top\partial_\phi T_\phi]d\rho,\\
\mathbf b(\theta,\phi)&=\int\E_\omega[\partial_\phi T_\phi^\top\bV_{\theta,\phi}(x,T_\phi)]d\rho,
\end{aligned}
\end{equation}
the square in \eqref{eq:projected-langevin} can be expanded as $h^\top G_T h-2h^\top\mathbf b+\int\E_\omega\|\bV_{\theta,\phi}(x,T_\phi)\|^2d\rho$, and that every minimizer solves the normal equations $G_T\dot\phi=\mathbf b$. We take the one of minimum norm, $$\dot\phi=G_T^+\mathbf b,$$ 
where for a symmetric positive semidefinite matrix $G$ with eigendecomposition $G=\sum_i\sigma_ie_ie_i^\top$, the Moore--Penrose pseudoinverse is $G^+:=\sum_{\sigma_i>0}\sigma_i^{-1}e_ie_i^\top$, and it equals $G^{-1}$ when $G$ is invertible. 

Notice that at the minimizer,  the constant term in the above expansion splits into the part that the family follows, $\|P_\phi\bV_{\theta,\phi}\|^2=\mathbf b^\top G_T^+\mathbf b$, and the projection defect $\|(I-P_\phi)\bV_{\theta,\phi}\|^2$, where $P_\phi$ is the orthogonal projection in $L^2(\omega\otimes\rho;\R^n)$ onto the sample velocities $\partial_\phi T_\phi h$. For example, for the Gaussian sampler $T(\xi)=m+L\xi$ with $m$ and $L$ being only functions of $x$ and $\omega=\cN(0,I_n)$, one has $\partial T/\partial m=I$, $\partial T/\partial L_{ij}=e_i\xi_j$, $G_T=I$ and $s_\phi(x,m+L\xi)=-L^{-\top}\xi$, and \eqref{eq:projected-langevin} gives
\[
\dot m=\E_\omega\big[\nabla_u\bH_\theta(x,m+L\xi)\big],\qquad
\dot L=\E_\omega\big[\nabla_u\bH_\theta(x,m+L\xi)\,\xi^\top\big]+\lambda L^{-\top}.
\]

Notice that without amortization, a grid-free scheme would keep particles at each sampled state and
run the Langevin dynamics there; since the sampled states change with the occupation
measure, these particles would have to be regenerated at every update. The projected
flow replaces the state-indexed family of Langevin dynamics by a single flow in $\phi$,
and states interact only through the shared parameters in $G_T$. The price is the
component $(I-P_\phi)\mathbf V$ of the velocity that the family cannot represent, which
the score loss monitors.

\paragraph{Critic flow.}\label{sec: critic}
Define the current policy source and its value residual by
\begin{equation}\label{eq:source}
g_\phi(x;\theta)=\E_{u\sim\pi_\phi(x,\cdot)}
\big[\bH_\theta(x,u)-\lambda\ln\pi_\phi(x,u)\big],\qquad
\cE^\phi_\theta=g_\phi-\beta v_\theta.
\end{equation}
Project the value dynamics in \eqref{eq:v-dynamics-r} with the same idea as for the policy:
\begin{equation}\label{eq:projected-flow}
\dot\theta=\argmin_h\int_{\R^d}|\nabla_\theta v_\theta(x)^\top h-\cE^\phi_\theta(x)|^2\rho(dx),\qquad
{G_\theta}=\int \nabla_\theta v_\theta \nabla_\theta v_\theta(x)^\top d\rho.
\end{equation}
With $\mathbf r=\int \nabla_\theta v_\theta\cE^\phi_\theta\,d\rho$ and selecting the minimum-norm solutions gives the coupled parameter flows
\begin{equation}\label{eq:joint-flow}
\boxed{\quad
\dot\phi=G_T^+\mathbf b,\qquad
\dot\theta={G_\theta^+}\mathbf r.
\quad}
\end{equation}
If the sampler tangent space contains the Langevin velocity and the critic tangent space contains the value residual, the projections are exact. Section \ref{sec: LQ} verifies this for affine Gaussian actors and quadratic critics.

\paragraph{Relation to fitted value iteration and temporal-difference (TD) learning.}
By Definition \ref{def:shift}, an increment $d\tau$ of the pointwise continuous iteration updates the value by the corresponding Hamiltonian under the newest policy, while the policy moves toward the Gibbs law of the current value. It is a continuous-time version of temporal difference learning in \citep{bertsekas1996neuro} which guaranteed partial improvement through the Langevin step in policies. The proposed critic flow replaces each continuous update by a regression. For a critic linear in its parameters, $v_\theta=\theta^\top a$ with fixed features $a(x)$, an explicit Euler step of length $\Delta\tau$ of \eqref{eq:projected-flow} gives
\[
v_{\theta_{k+1}}=\Pi_\rho\big[v_{\theta_k}+\Delta\tau\,\cE^\phi_{\theta_k}\big]\qquad\rho\text{-a.e.},
\]
where $\Pi_\rho$ is the orthogonal projection of $L^2(\rho)$ onto the span of the features. By It\^o's formula, the update of $v_{\theta_k}$ over the physical time $\Delta\tau$ under the current actor is $v_{\theta_k}(x)+\E_x^{\pi_\phi}\int_0^{\Delta\tau}e^{-\beta t}\cE^\phi_{\theta_k}(X_t)dt$. The Euler step of this is similar to one step of fitted value iteration \citep{munos2008finite} for the evaluation equation \eqref{eq:policy-evaluation} of the current actor, while the update is computed from the generator instead of from sampled next states. For a nonlinear class, \eqref{eq:projected-flow} projects the velocity of the value dynamics onto the tangent space of the family. This resembles the Dirac--Frenkel time-dependent variational principle \citep{lubich2008quantum}, used to evolve network parameters for evolution equations by \citet{du2021evolutional} and in the neural Galerkin schemes of \citet{bruna2024neural}. Similar to the fitted $Q$-iteration in continuous action spaces, where the maximization over actions is replaced by a search over a policy class \citep{antos2007fitted}, our actor moves toward the Gibbs law of the critic.

At a stationary point, $\mathbf r=\int \nabla_\theta v_\theta\cE^\phi_\theta\,d\rho=0$, so the value residual is $L^2(\rho)$-orthogonal to the tangent space. For a linear class this is the Galerkin approximation of \eqref{eq:policy-evaluation} with weight $\rho$, as in the Galerkin approximations of the generalized HJB equation within policy iteration of \citet{beard1997galerkin}. Its fixed-point $v_\theta=\Pi_\rho[g_\phi(\cdot;\theta)/\beta]$ is the counterpart of the projected Bellman equation solved by temporal-difference learning with linear function approximation \citep{tsitsiklis1997analysis}. With $\rho=\rho^\phi$ from \eqref{eq:occupation-measure}, It\^o's formula applied to the exploratory state process $X$ under $\pi_\phi$ with $X_0\sim\mu$ implies
\[
\mathbf r=\beta\,\E_\mu\int_0^\infty e^{-\beta t}\nabla_\theta v_\theta(X_t)\big[dv_\theta(X_t)+(r_{\pi_\phi}-\beta v_\theta)(X_t)\,dt\big].
\]
This is the expected continuous-time temporal-difference update \citep{doya2000reinforcement} with test functions $\nabla_\theta v_\theta$, and $\mathbf r=0$ coincide with the martingale orthogonality condition in the sense of \citet{jia2022policy_evaluation}, evaluated from the generator rather than along sampled paths as in martingale and backward-SDE methods \citep{jia2022policy_evaluation, han2018solving, beck2019machine}. 

The weighting $\rho$ plays the role of the sampling distribution of temporal-difference learning, which converges for linear classes under on-policy sampling and can diverge under other sampling distributions or with nonlinear classes \citep{baird1995residual, tsitsiklis1997analysis}. Residual-gradient methods \citep{baird1995residual} instead minimize a squared residual, $\int|\cE^\phi_\theta|^2d\rho$ or $\int|\cR_\theta|^2d\rho$ with the HJB residual $\cR_\theta$ of \eqref{eq:residuals}, is the least-squares collocation of physics-informed and deep Galerkin solvers \citep{raissi2019physics, sirignano2018dgm}, used for exploratory HJB equations within policy iteration by \citet{kim2026physics}. Finally, notice that actor--critic algorithms with a temporal-difference critic are usually analyzed on two time scales, with the critic faster than the actor \citep{konda2003actor}. The joint flow \eqref{eq:joint-flow} runs both in one fictitious time because the pointwise iteration it projects does, and Lemma \ref{lemma:fisher} shows that the exact coupled flow improves monotonically when the critic starts from the exact value of the initial policy.

\subsection{The algorithm}\label{sec: algorithm}
We next specifies the algorithm that implement the score-Hamiltonian iteration introduced above. At each iteration we sample $N$ states and $K$ controls per state. For actor samples $u_{ij}=T_\phi(x_i,\xi_{ij})$, put
$J_{ij}=\partial_\phi T_\phi(x_i,\xi_{ij})$ and
$V_{ij}=\bV_{\theta,\phi}(x_i,u_{ij})$. The empirical actor system is
\begin{equation}\label{eq:empirical-actor}
\widehat G_T=\frac1{NK}\sum_{i,j}J_{ij}^\top J_{ij},\qquad
\widehat{\mathbf b}=\frac1{NK}\sum_{i,j}J_{ij}^\top V_{ij}.
\end{equation}
For the critic, let $a_i=\nabla_\theta v_\theta(x_i)$ and let $\widehat g_i$ estimate \eqref{eq:source} using the updated actor. Then
\begin{equation}\label{eq:empirical-critic}
{\widehat G_\theta}=\frac1N\sum_i a_i a_i^\top,\qquad
\widehat{\mathbf r}=\frac1N\sum_i a_i\big(\widehat g_i-\beta v_\theta(x_i)\big).
\end{equation}
Exact action expectations replace sample averages whenever available. In particular, Gaussian entropy is analytic and for an invertible sampler it is computed from the base entropy and the expected log-Jacobian determinant.

\begin{center}
\refstepcounter{shpialgorithm}\label{alg:shpi}
\fbox{\begin{minipage}{0.94\linewidth}
\textbf{Algorithm \theshpialgorithm. Score-Hamiltonian policy iteration}\par\smallskip
\textbf{Input:} actor $T_\phi$, critic $v_\theta$, state-sampling rule, batch sizes $N,K$, step $\Delta\tau$, and damping parameters $\eta_a,\eta_c\geq0$.\par
\textbf{Initialize:} choose $\phi_0,\theta_0$, initialize the critic to the value of the initial policy. \par\smallskip
\textbf{Repeat:}
\begin{enumerate}
\item \textbf{Sample states.} Draw or refresh $x_1,\ldots,x_N$ from the state-sampling rule.
\item \textbf{Compute the target.} Evaluate $v_\theta$, $v_{\theta,x}$ and the the Hessian term in the Hamiltonian, at the batch states.
\item \textbf{Update the actor.} Draw $\xi_{ij}\sim\omega$, compute $u_{ij}$, $V_{ij}$ and \eqref{eq:empirical-actor}, and set
\[
\phi\leftarrow\phi+\Delta\tau(\widehat G_T+\eta_a I)^+\widehat{\mathbf b}.
\]
\item \textbf{Update the critic.} Using fresh controls from the updated actor, or exact action expectations, compute $\widehat g_i$ and \eqref{eq:empirical-critic}, and set
\[
\theta\leftarrow\theta+\Delta\tau({\widehat G_\theta}+\eta_c I)^+\widehat{\mathbf r}.
\]
If used, apply the derivative-residual correction described in Section \ref{sec: three flows}.
\item \textbf{Monitor.} On a validation batch, record the actor and critic residuals in \eqref{eq:residuals}, together with the parameter-update sizes. Stop at the prescribed tolerance or computational budget.
\end{enumerate}
\textbf{Output:} $v_\theta$ and the policy sampled by $u=T_\phi(x,\xi)$.
\end{minipage}}
\end{center}

The actor-then-critic order is a first-order sequential discretization of \eqref{eq:joint-flow}. Damping stabilizes ill-conditioned Gram systems and it changes the exact projection. A practical state-sampling rule draws $X_0\sim\mu$, simulates the exploratory dynamics under the current policy, and stops at an independent $T\sim\mathrm{Exp}(\beta)$. Exact stopped states have the normalized discounted occupation law
\begin{equation}\label{eq:occupation-measure}
\rho^\phi=\beta\int_0^\infty e^{-\beta t}\mathbb P_\mu(X_t^{\pi_\phi}\in\cdot)dt,
\end{equation}
which helps sampling states that are visited more often under the current policy. The experiments of Section \ref{sec: numerics} refresh a state pool from the current policy's occupation measure (Section \ref{sec: num B}), weight a fixed set of nodes by it (Section \ref{sec: num DW}), or use a fixed reference-based pool (Section \ref{sec: num flows}), and may mix it with a wider exploration distribution. Exact agreement with the pointwise iteration requires exact initial policy evaluation, exact expectations and no projection error, and the finite-batch algorithm is an approximation to these flows.

\subsection{Residuals Analysis}\label{sec: two residuals}
We report three residuals: the actor residual $\cA^\phi_\theta$ compares the policy with the Gibbs law of the current critic, the critic residual $\cE^\phi_\theta$ checks the value equation under the current policy, and  the HJB residual $\cR_\theta$ characerizes the optimality condition:
\begin{equation}\label{eq:residuals}
\begin{aligned}
\cA^\phi_\theta(x)&=\E_{\pi_\phi(x,\cdot)}\|\bV_{\theta,\phi}(x,u)\|^2
=\lambda^2\Fisher(\pi_\phi(x,\cdot)\|\Gamma[v_\theta](x,\cdot)),\\
\cE^\phi_\theta(x)&=g_\phi(x;\theta)-\beta v_\theta(x),\\
\cR_\theta(x)&=\lambda\ln\int_Ue^{\bH_\theta(x,u)/\lambda}du-\beta v_\theta(x).
\end{aligned}
\end{equation}
%The score loss is
%\begin{equation}\label{eq:score-loss}
%\cL(\phi;v_\theta)=\int\cA^\phi_\theta\,d\rho.
%\end{equation}
%For multimodal targets, we also evaluate the score difference under a fixed reference density $\nu$ covering the control regions of interest:
%\begin{equation}\label{eq:reference-certificate}
%\cL^\nu(\phi;v_\theta)=\int\int_U\|\bV_{\theta,\phi}(x,u)\|^2\nu(u)du\,\rho(dx),
%\end{equation}
%which reveals errors in control regions poorly sampled by the actor. Small projected updates alone do not imply small residuals.
At a fixed critic and sampling law, the population actor dissipates the Gibbs free energy at its projected rate:
\begin{equation}\label{eq:projected-dissipation}
\frac{d}{d\tau}\lambda\int\mathrm{KL}(\pi_\phi\|\Gamma[v_\theta])d\rho
=-\mathbf b^\top G_T^+\mathbf b=-\|P_\phi\bV_{\theta,\phi}\|^2,
\end{equation}
where $P_\phi$ is the orthogonal projection onto sample velocities $\partial_\phi T_\phi h$ in $L^2(\omega\otimes\rho)$. The next proposition decomposes the HJB residual into two components, and characterize the estimation error in terms of residuals. 

\begin{proposition}\label{prop:two-residuals}
(i) For positive smooth actor densities with finite entropy and integrable Hamiltonian,
\begin{equation}\label{eq:R-E-KL}
\cR_\theta=\cE^\phi_\theta+\lambda\,\mathrm{KL}(\pi_\phi\|\Gamma[v_\theta]),
\end{equation}
and $\cA^\phi_\theta=\cE^\phi_\theta=0$ is equivalent to that $v_\theta$ solves the HJB equation and $\pi_\phi=\Gamma[v_\theta]$ is the optimal policy $\rho$-a.e. 

(ii) If $-\nabla^2_{uu}\bH_\theta(x,u)\succeq\mu(x)I$ with $\mu(x)>0$, then
\begin{equation}\label{eq:KL-Fisher}
\lambda\,\mathrm{KL}(\pi_\phi\|\Gamma[v_\theta])\leq\frac{\cA^\phi_\theta}{2\mu(x)}.
\end{equation}

(iii) If $\mu_*:=\inf_x\mu(x)>0$, the learned policy is admissible and along the state process $X$ under the stationary policy $\pi_\theta$,

(a) $e^{-\beta T}\E_x^\pi[v_\theta(X_T)]\longrightarrow0$, as $T\rightarrow \infty$,

(b) $e^{-\beta t}v_\theta^\top\bar\sigma(X_t)$ is square integrable, and 

(c) the residual are absolutely integrable,

Then
\begin{equation}\label{eq:suboptimality}
0\leq V(x)-J(x,\pi_\phi)\leq
\frac2\beta\|\cE^\phi_\theta\|_\infty+
\frac1{2\beta\mu_*}\|\cA^\phi_\theta\|_\infty.
\end{equation}
\end{proposition}
\begin{proof}
Substitute $\ln\Gamma[v_\theta]=\bH_\theta/\lambda-\ln \int_Ue^{\bH_\theta/\lambda}du$ into the definition of $\cR_\theta$ yields \eqref{eq:R-E-KL}, and \eqref{eq:projected-dissipation} impies the optimality condition. 

Sicne $-\bH/\lambda$ has Hessian bounded below by $(\mu/\lambda)I_n$, the Bakry--\'Emery logarithmic Sobolev inequality \citep{bakry1985diffusions} therefore gives
\[
\mathrm{KL}(\pi_\phi\|\Gamma[v_\theta])\leq\frac{\lambda}{2\mu(x)}
\Fisher(\pi_\phi\|\Gamma[v_\theta]).
\]Multiplying by $\lambda$ on both sides implies \eqref{eq:KL-Fisher}

Finally, applying It\^o's formula to $e^{-\beta t}v_\theta(X_t)$ and use the transversality and the fact the $e^{-\beta t}J(X_t,\pi_\theta)$ is a martingale, 
\begin{align}
J(x,\pi_\phi)-v_\theta(x)&=\E_x^{\pi_\phi}\int_0^\infty e^{-\beta t}\cE^\phi_\theta(X_t)dt,\label{eq:J-minus-v}
\end{align}
In particular, if $\pi_\phi = \pi^*$, then $J(x,\pi_\phi) = V(x)$, and 
applying it to $\pi^*$ and $\cE^\phi_\theta$ is bounded by $\cR_\theta$ by \eqref{eq:R-E-KL}, it follows that
\begin{align}
V(x)-v_\theta(x)&\leq\E_x^{\pi^*}\int_0^\infty e^{-\beta t}\cR_\theta(X_t)dt.\label{eq:V-minus-v}
\end{align}
Taking the difference between \eqref{eq:V-minus-v} and \eqref{eq:J-minus-v} implies \eqref{eq:suboptimality}.
\end{proof}

\subsection{Derivative corrections}\label{sec: three flows}
%\paragraph{Hamiltonian derivatives.}
%In the above algorithm, the actor iteration uses $\nabla_u\bH_\theta$, not a dense Hessian matrix alone. In the special case where both the drift and diffusion of the state dynamics are linear in the control variable and independent of the state, i.e. $b=b_0+Bu$ and $\sigma=\sigma_0+\sum_{j=1}^n\sigma_j u_j$, then for any symmetric $P$,
%\begin{equation}\label{eq:hessian-form}
%\begin{gathered}
%\nabla_u\bH=\nabla_uf+B^\top p+m(P)+M(P)u,\\
%m_i(P)=\tr(\sigma_i^\top P\sigma_0),\qquad
%M_{ij}(P)=\tr(\sigma_i^\top P\sigma_j).
%\end{gathered}
%\end{equation}
%The coefficients use $k(n+1)$ Hessian-vector products and are reused across the action batch, while if the diffusion is control-independent, no Hessian enters the actor target. For fixed sampling budgets and polynomial-size parameterizations, these evaluations and the Gram solvers have polynomial arithmetic cost, while conditioning, approximation accuracy and sampling remain separate computational issues.

%\paragraph{Gibbs structure.}\label{sec:implicit-gibbs}
%An alternative to a learned sampler is to sample $\Gamma[v_\theta]$ directly using its analytic score. Then $g_\phi=\lambda\ln Z_\theta$ and the actor step is replaced by Gibbs sampling; finite-chain and normalizer errors must be included in the implementation. 

The policy depends on critic derivatives, whose accuracy need not follow from a small sampled scalar residual of the value iteration. To improve them, the experiments can optionally follow the critic step by projections of the gradient and Laplacian lines of \eqref{eq:v-dynamics-r} onto the critic family, whose targets are the derivative residuals
\begin{equation}\label{eq:residual-x-xx}
\cE^{(1)}_\theta=\nabla_xg_\phi-\beta\nabla_xv_\theta,\qquad
\cE^{(2)}_\theta=\Delta g_\phi-\beta\Delta v_\theta.
\end{equation}
The parameter velocity $h$ minimizes 
\begin{equation}\label{eq:weights}
\frac12\int(\alpha_1\|\partial_\theta\nabla_xv_\theta\,h-\cE^{(1)}_\theta\|^2+\alpha_2|\partial_\theta\Delta v_\theta\,h-\cE^{(2)}_\theta|^2)d\rho, 
\end{equation}
with chosen nonnegative weights and a small damping term, and $\theta\leftarrow\theta+\Delta\tau\,h$, while for the undamped system, $h$ is $\beta$ times the Gauss--Newton step for the loss $\frac12\int(\alpha_1\|\cE^{(1)}_\theta\|^2+\alpha_2|\cE^{(2)}_\theta|^2)d\rho$ when the entire source $g_\phi$ and its spatial derivatives are frozen in the parameter differentiation. This is a frozen-source surrogate: recomputing the source after the update does not in general preserve a residual-descent guarantee. 

\section{Exactness in the linear-quadratic class}\label{sec: LQ}
In this section we analyze the score-Hamiltonian iteration in the linear-quadratic models and show that it recovers the continuous iteration in \cite{feng2026continuous} and converge to the optimal solution with reasonable parametric families for the policy and value. Let $b(x,u)=Ax+Cu$, $\sigma(x,u)=\sigma_0$, and
$f(x,u)=-\frac12(x^\top Qx+u^\top Ru)-q^\top x-r^\top u$, with $Q\succeq0$ and $R\succ0$. With the ansatz $v(x)=-\frac12x^\top\Pi x-p^\top x-c$, $\Pi=\Pi^\top$, the target score and Gibbs policy are
\begin{equation}\label{eq:LQ-policy}
\begin{aligned}
\nabla_u\bH(x,u,v_x,v_{xx})&=-Ru-r-C^\top(\Pi x+p),\\
\Gamma[v](x,\cdot)&=\cN\big(-R^{-1}(C^\top\Pi x+C^\top p+r),\ \lambda R^{-1}\big).
\end{aligned}
\end{equation}
Substitution into the HJB equation \eqref{eq:HJB-relaxed} implies the following system
\begin{equation}\label{eq:ARE}
\begin{gathered}
A^\top\Pi+\Pi A-\beta\Pi-\Pi CR^{-1}C^\top\Pi+Q=0,\\
\big(\beta I-(A-CK)^\top\big)p=q-K^\top r,\qquad K:=R^{-1}C^\top\Pi,\\
c=\frac{\tr[\sigma_0\sigma_0^\top\Pi]-(C^\top p+r)^\top R^{-1}(C^\top p+r)-\lambda n\ln(2\pi\lambda)+\lambda\ln\det R}{2\beta}.
\end{gathered}
\end{equation}
Under stabilizability and detectability for the discounted system, the value solution is the Riccati root for which $A-CK-\frac\beta2I$ is Hurwitz, together with the verification and transversality conditions.

Use a quadratic critic and an affine Gaussian sampler,
\begin{equation}\label{eq:LQ-param}
\begin{gathered}
v_\theta(x)=-\frac12x^\top\Pi_\theta x-p_\theta^\top x-c_\theta,\qquad
T_\phi(x,\xi)=-K_\phi x-k_\phi+L_\phi\xi,\quad \xi\sim\cN(0,I_n),\\
\pi_\phi(x,\cdot)=\cN(-K_\phi x-k_\phi,\Sigma_\phi),\qquad \Sigma_\phi=L_\phi L_\phi^\top,
\end{gathered}
\end{equation}
where $\Pi_\theta$ is symmetric and $L_\phi$ is a free invertible matrix. The critic residual is quadratic, and the {Langevin velocity} at $T_\phi(x,\xi)$ is affine in $(x,\xi)$. Both therefore lie in their respective tangent spaces. Dropping parameter subscripts, the projected critic equations are
\begin{equation}\label{eq:lyapunov}
\begin{gathered}
\dot\Pi=Q+K^\top RK+(A-CK)^\top\Pi+\Pi(A-CK)-\beta\Pi,\\
\dot p=(A-CK-\beta I)^\top p+K^\top(Rk-r)-\Pi Ck+q,
\end{gathered}
\end{equation}
\begin{equation}\label{eq:LQ-constant-flow}
\dot c=\frac12\tr[\sigma_0\sigma_0^\top\Pi]+\frac12\big(k^\top Rk+\tr[R\Sigma]\big)-k^\top(C^\top p+r)
-\frac\lambda2\ln\big((2\pi e)^n\det\Sigma\big)-\beta c,
\end{equation}
and the actor equations are
\begin{equation}\label{eq:LQ-langevin}
\begin{gathered}
\dot K=C^\top\Pi-RK,\qquad \dot k=C^\top p+r-Rk,\\
\dot L=\lambda L^{-\top}-RL,\qquad \dot\Sigma=2\lambda I-R\Sigma-\Sigma R.
\end{gathered}
\end{equation}
At a joint stationary point of these equations recover \eqref{eq:LQ-policy} and \eqref{eq:ARE}.

\begin{proposition}
\label{prop:lq-single-tau}
Let $\pi_0(x,\cdot)=\cN(-K_0x-k_0,\Sigma_0)$, $\Sigma_0\succ0$, with a finite value $v^0=J(\cdot,\pi_0)$. Initialize \eqref{eq:LQ-param} with $v_{\theta_0}=v^0$ and $\pi_{\phi_0}=\pi_0$.
Let $\alpha_0$ be the the largest real part of the eigenvalues of $A-CK_0$.
Assume \(\beta>2\alpha_0\), $\rho$ has finite fourth moment and the Gram matrices are non-singular, then the undamped projections reproduce the pointwise iteration:
\begin{equation*}
v_{\theta_\tau}=v^\tau,\qquad \pi_{\phi_\tau}(x,\cdot)=\pi_\tau(x,\cdot)\quad\text{for every }x.
\end{equation*}
The explicit Euler updates also coincide if the action expectations are exact, and both schemes use the same step, update order and coordinates.
\end{proposition}
\begin{proof}
Since the actor velocity belongs to the sampler's tangent space and the critic residual is quadratic, the nonsingular Gram matrices therefore recover their coefficients exactly in \eqref{eq:lyapunov}--\eqref{eq:LQ-langevin}. The pointwise continuous iteration in \cite{feng2026continuous} preserves the same classes and solves the same locally Lipschitz coefficient system. Thus the uniqueness of its solution implies that the two iterations coincide. The claim for the Euler maps follows similarly.
\end{proof}

\section{Numerical experiments}\label{sec: numerics}
In this section, we test the accuracy of our derivative accuracy in three numerical experiments, featuring a high-dimensional state, learning of a bimodal control law, and coupled actor--critic learning at $d=n=16$, respectively. All three examples use the known control-affine model
\begin{equation}\label{eq:affine-setting}
b(x,u)=b_0(x)+B(x)u,\qquad \sigma(x,u)=\sigma(x),\qquad f(x,u)=f_0(x)-c(u).
\end{equation}
We write $z_v(x)=B(x)^\top\nabla_xv(x)$ for the \emph{control-relevant value gradient}. 

The updates use Gram projections and derivative corrections are specified in Section~\ref{sec: num flows}. All runs use fixed update budgets. Relative errors below are empirical root mean square errors divided by the reference root mean square. For a multivariate policy, TV is the mean total variation distance between corresponding one-dimensional control marginals, not the total variation distance between joint laws. Learners use sampled states while reference calculations and control integrals use grids or quadrature.

\subsection{Interacting double wells with a trainable critic}\label{sec: num flows}
This experiment tests the value projection and its derivative corrections at $d=100$ and $n=6$. The control policy is the Gibbs law determined by the current critic, so there is no separately parametrized actor. We use
\[
c(u)=\sum_{i=1}^{6}\frac{\kappa}{4}(u_i^2-a^2)^2
+\frac{\gamma}{2}\Bigl(\sum_{i=1}^{6}u_i\Bigr)^2,
\qquad(\kappa,a,\gamma)=(2,1,0.3),\quad
(\lambda,\beta,s)=(0.35,0.5,0.5).
\]
The quartic terms create two wells in each control coordinate, while the last term couples the coordinates. The next lemma shows that we can construct any desired optimal value by choosing appropriate $f_0$.

\begin{lemma}\label{lemma:manufactured}
For a sufficiently smooth candidate $\bar v$, set $\bar z=B^\top\nabla_x\bar v$ and
\begin{equation}\label{eq:manufactured}
f_0(x)=\beta\bar v(x)-(Ax)\cdot\nabla_x\bar v(x)
-\tfrac{s^2}{2}\Delta\bar v(x)-\lambda\ln Z(\bar z(x)).
\end{equation}
If $Z(\bar z(x)) =\int_{\mathbb R^6}\exp\!\left(\frac{-c(u)+\bar z\cdot u}{\lambda}\right)du$ is finite, then $\bar v$ satisfies the relaxed HJB equation with Gibbs policy $\pi_{\bar z(x)} = \pi_z(u)=Z(z)^{-1}\exp\!\left(\frac{-c(u)+z\cdot u}{\lambda}\right)$. Under the admissibility, integrability and discounted transversality conditions (similar to Proposition \ref{prop:two-residuals}(iii)), this policy is optimal and $V=\bar v$.
\end{lemma}
\begin{proof}
For the uncontrolled diffusion, the control dependence of the Hamiltonian is $-c(u)+\bar z\cdot u$ and the Gibbs density is $\pi_{\bar z}$. The Gibbs variational identity turns its maximum into $\lambda\ln Z(\bar z)$. Algebraic calculation verifies that $\mathbb E_{\pi_{\bar z}}[\bH(x,u,\bar v_x,\bar v_{xx})-\lambda \ln\pi_{\bar z}] = \beta \bar v$. Then, since the Gibbs measure $\pi_{\bar z}$ satisfies the first order condition in the HJB equation,
$$\sup_{\pi\in\mathcal P(U)}\mathbb E_{\pi}[\bH(x,u,\bar v_x,\bar v_{xx})-\lambda \ln\pi] = \beta \bar v,$$
and $\bar v$ solves the HJB equation. The verification follows from standard argument.
\end{proof}
We take $b_0(x)=Ax$, $B$ constant, $\sigma=sI_d$, and construct a reference value with two-feature interactions:
\[
\begin{aligned}
A&=-0.6I+\tfrac12S_{\rm sk},\qquad P=I+0.3L^\top L/d,\\
\bar v(x)&=-\tfrac12x^\top Px+
\sum_{j=1}^{8}\varsigma_j
\tanh(w_j^\top x+b_j)\tanh(r_j^\top x+c_j).
\end{aligned}
\]
Here $S_{\rm sk}$ is skew-symmetric, $L\in\mathbb R^{32\times d}$, and the sixteen directions $\{w_j,r_j\}_{j=1}^8$ are orthonormal. Each interaction has mixed state derivatives, whereas the critic below uses a sum of single-feature activations. 

Lemma \ref{lemma:manufactured} implies that with value gradient $\bar z(x)=B^\top\nabla_x\bar v(x)$,
\begin{equation}\label{eq:manufactured}
f_0(x)=\beta\bar v(x)-(Ax)\cdot\nabla_x\bar v(x)
-\tfrac{s^2}{2}\Delta\bar v(x)-\lambda\ln Z(\bar z(x)).
\end{equation}
makes $\bar v$ satisfy the HJB equation and the reference optimal policy is $\pi_{\bar z(x)}$ (the conditions for verification is satisfied by the bounded nonlinear interactions which preserve the quadratic tails).

The critic family is
\[
\begin{aligned}
v_\theta(x)&=a_0+a^\top x+\tfrac12x^\top Q_\theta x
+\sum_{k=1}^{64}c_k\tanh(W_k^\top x+q_k),\\
Q_\theta&=\operatorname{diag}(\delta)+\tfrac12(M^\top K+K^\top M),
\qquad M,K\in\mathbb R^{24\times d}.
\end{aligned}
\]
Amplitudes, hidden directions and biases are trained and no information about the reference direction is embedded in the critic. This family has $11\,529$ parameters at $d=100$. It starts at $v_\theta=0$. The Gibbs policy at each step is $\pi_{z_{v_\theta}(x)}$, with $\ln Z$, moments and control marginals evaluated by structured quadrature.

We compare the results with method V, the projected value updates with those from method V$X$, which adds a gradient-residual projection, and V$XX$, which adds gradient- and Laplacian-residual projections after each value step, as described in Section \ref{sec: three flows}. Their derivative weights in \eqref{eq:weights} are $(\alpha_1,\alpha_2)=(10,0)$ and $(10,0.3)$, respectively. Method R is a separate damped Gauss--Newton minimization of the HJB residual. All methods use the same reference-based pool of $3072$ occupation states and $1024$ radius-matched Gaussian states, $256$-state batches, and $80$ updates with $\Delta\tau=0.2$. The damping is $10^{-2}$ times the initial mean batch-kernel eigenvalue. Monitoring uses $400$ states and final evaluation uses $600$ independent states.

\begin{table}[!htbp]
\centering\small\setlength{\tabcolsep}{5pt}
\begin{tabular}{lrrrr}
Method & Marginal TV & Value-gradient error & Value error & HJB RMS\\\hline
R   & $0.06115$ & $29.95\%$ & $5.43\%$ & $1.025$\\
V   & $0.05580$ & $27.62\%$ & $4.98\%$ & $0.950$\\
V$X$  & $0.00935$ & $4.74\%$ & $16.39\%$ & $0.874$\\
V$XX$ & $0.00935$ & $4.74\%$ & $16.64\%$ & $0.887$
\end{tabular}
\caption{Trainable-critic results after $80$ updates on $600$ independent states. The value-gradient error compares $B^\top\nabla v_\theta$ with $\bar z$. TV averages the six control marginals.}\label{tab:flows}
\end{table}
Table \ref{tab:flows} shows that the derivative corrections substantially improve the control-relevant gradient and the marginal policies, although their value error is larger than V's at this update budget. On the fixed $400$-state monitoring set, V$X$ marginal TV falls from $0.03237$ at update $5$ to $0.00935$ at update $80$; V$XX$ follows closely. Figure~\ref{fig:5-1-marginals} shows all six marginals at a predetermined state, and Figure~\ref{fig:5-1-convergence} shows the monitoring curves.

\begin{figure}[!htbp]
\centering\includegraphics[width=\linewidth]{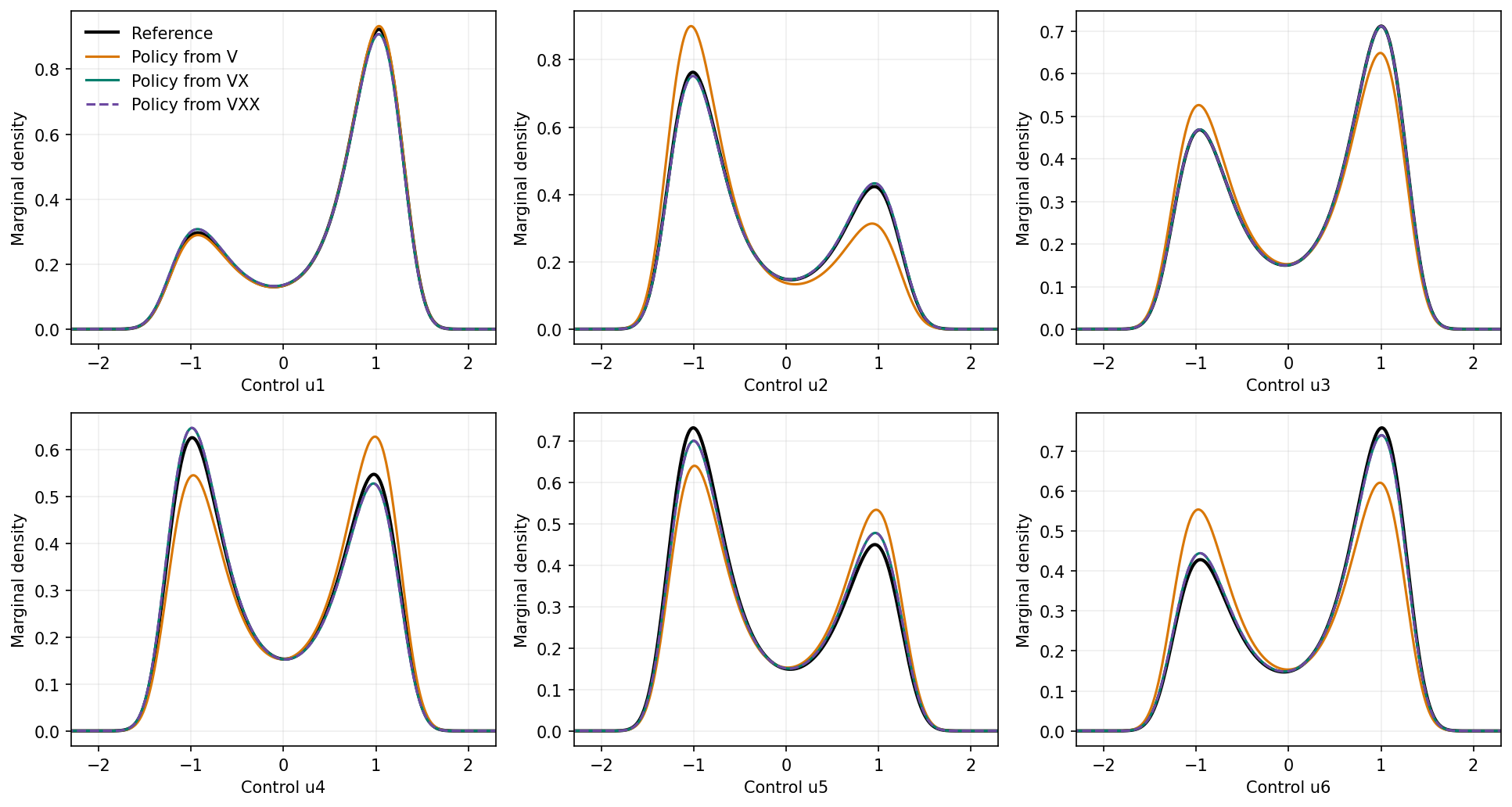}
\caption{Six control marginals at a predeterminted state. Black is the optimal reference and the colored curves use Gibbs policies from the trained V, V$X$ and V$XX$ critics.}\label{fig:5-1-marginals}
\end{figure}
\begin{figure}[!htbp]
\centering\includegraphics[width=\linewidth]{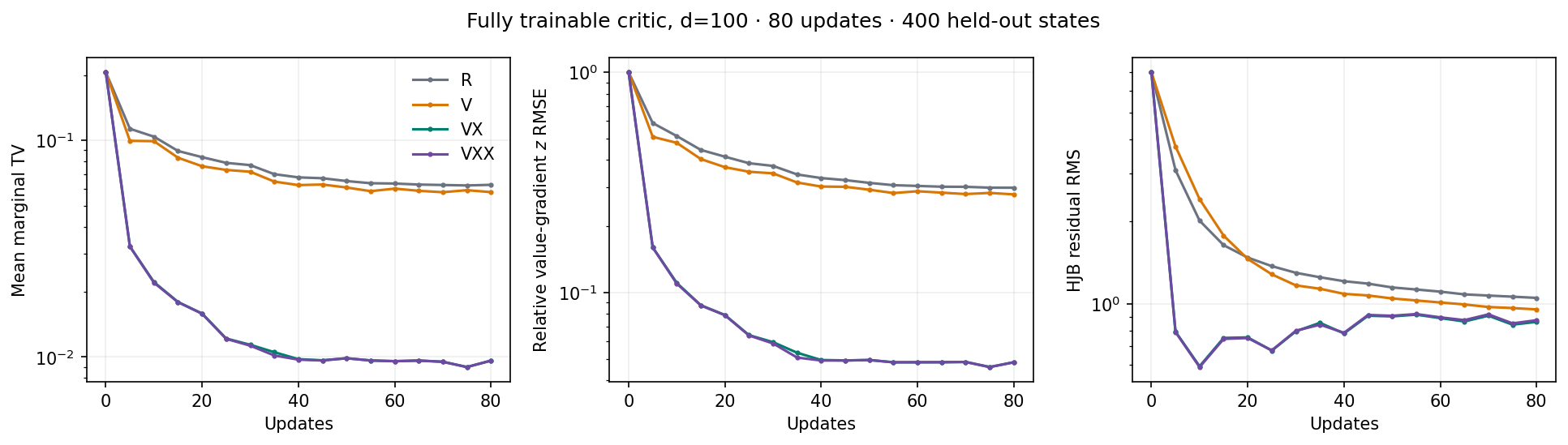}
\caption{Mean marginal TV, relative control-relevant value-gradient error and HJB-residual RMS on the $400$ monitoring states.}\label{fig:5-1-convergence}
\end{figure}

\subsection{Learning a bimodal control law}\label{sec: num DW}
In this example we test the our algorithm in a one-dimensional model ($d=n=1$) where we learn the double-well control with our actor projection iteration.
\[
b(x,u)=-x+0.5u,\qquad \sigma(x)=0.5x,\qquad
f(x,u)=-\tfrac12x^2-0.3x-(u^2-1)^2,\qquad
(\lambda,\beta)=(0.25,0.5).
\]
A numerical Markov-chain with $121$ nodes in $[-3,3]$, boundary reflection and control quadrature are used to calculate the reference. The same state chain supplies policy values and occupation weights. The reflection is only used for numerical reference, and is not used in the the control iteration. Policy values are averaged under $\mu$, uniform on $101$ nodes in $[-2.5,2.5]$.

The critic is a degree-six Legendre expansion,
\[
v_\theta(x)=\sum_{j=0}^{6}\theta_j\psi_j(x),\qquad
\psi_j(x)=\sqrt{2j+1}\,P_j(x/2.6),\qquad
z_{v_\theta}(x)=0.5v_\theta'(x).
\]
We compare a Gaussian actor with cubic mean, with only the first two terms in the following, and the monotone actor, by adding the third term: 
\[
T_\phi(x,\xi)=\alpha_0(x)+e^{\alpha_1(x)}\xi+
\sum_{k=1}^{3}e^{w_k(x)}\tanh\!\left(\frac{\xi-\ell_k(x)}{0.15}\right),
\qquad \xi\sim\mathcal N(0,1).
\]
Its eight coefficient functions are cubic in $x$ (32 parameters). Since $\partial_\xi T_\phi>0$, its density, score and entropy follow by change of variables. 

We perform two experiments in this model. In the first, similarly Section \ref{sec: num flows}, the critic is iterated using its implcit Gibbs, with $3000$ projected value updates of size $0.05$. Then with this critic frozen, the Gaussian and monotone actors take $6000$ and $20\,000$ projected Langevin updates of size $0.01$, using their current finite-chain occupation laws and deterministic control-noise quadrature.

\begin{table}[!htbp]
\centering\small
\begin{tabular}{lrr}
Policy & TV at $x=-0.4$ & $J_\mu$\\\hline
Gaussian actor & $0.510$ & $-0.818$\\
Monotone actor & $0.094$ & $-0.452$\\
Critic Gibbs law & $0.005$ & $-0.433$
\end{tabular}
\caption{Frozen-critic policies compared with the discrete reference, whose value is $J_\mu=-0.4328$. TV integrates the conditional density discrepancy on the control grid.}\label{tab:bimodal-frozen}
\end{table}
Table \ref{tab:bimodal-frozen} and Figure \ref{fig:5-2-conditional} show the results at the state $x=0.4$. The Gaussian concentrates in one well at $x=-0.4$, while the monotone actor covers both. The critic Gibbs law provides a separate implicit-policy comparison. Figure~\ref{fig:5-2-conditional} displays these conditional densities across states. The finite grid and reflecting boundaries also affect the reference, so the table measures agreement with this numerical benchmark.

In the second experiment, we first fit the monotone actor to the symmetric law corresponding to zero control-relevant value gradient for $2000$ updates of size $0.01$. Starting from that actor and a zero critic, we then run $2000$ coupled actor-then-critic updates of size $0.01$. At $x=1.5$, TV between the actor and its moving critic-Gibbs target rises from $0.093$ to $0.440$ before falling to $0.089$; $J_\mu$ improves from $-0.530$ to $-0.454$. The final critic's maximum discrepancy from the discrete reference on $[-2.5,2.5]$ is $0.059$. Figure~\ref{fig:5-2-coupled} shows the changing target and value.

Finally, as a separate check, exact policy evaluation replaces the critic at each actor step with that actor's value on the finite chain, which separates actor-family limitations from critic-fitting error on the chosen discretization. With step size $0.01$, the Gaussian and monotone policies reach $J_\mu=-0.818$ and $-0.453$ after $1500$ and $2500$ steps, respectively.

\begin{figure}[!htbp]
\centering\includegraphics[width=\linewidth]{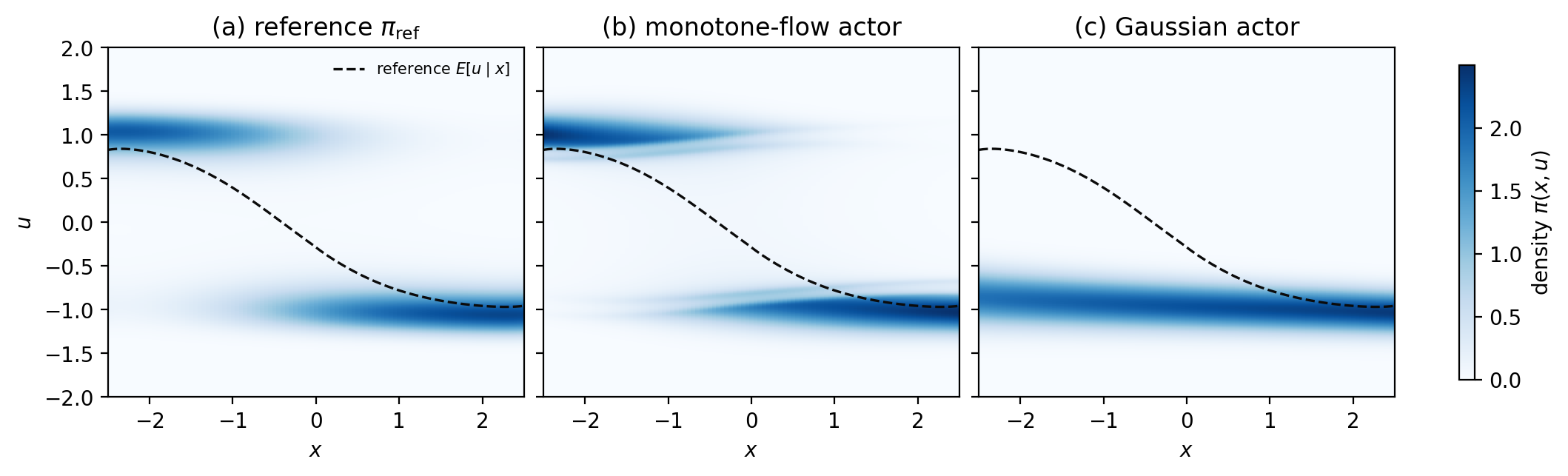}
\caption{Reference, monotone-flow and Gaussian conditional densities. The dashed curve in each panel is the reference conditional mean.}\label{fig:5-2-conditional}
\end{figure}
\begin{figure}[!htbp]
\centering\includegraphics[width=\linewidth]{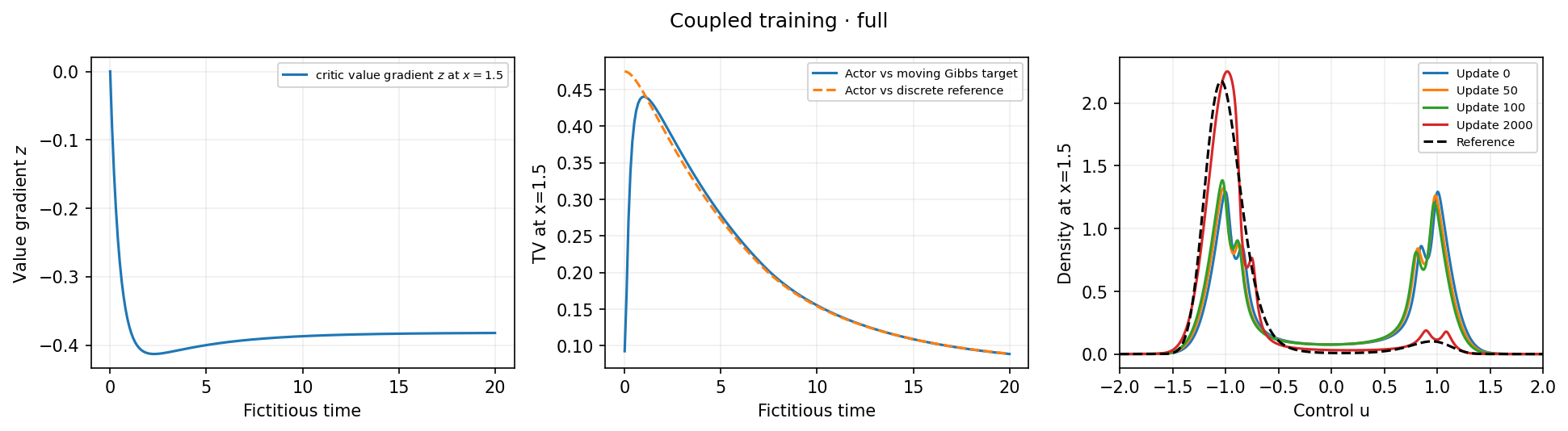}
\caption{Coupled monotone run at $x=1.5$: the critic's control-relevant value gradient, actor TV from its moving Gibbs target and the discrete reference, and the policy density updates}\label{fig:5-2-coupled}
\end{figure}
\subsection{Coupled learning of sixteen double-well controls}\label{sec: num B}
In this example we take $d=n=16$ and a double-well penalty in every control coordinate. Let $O$ be orthogonal, $A=O\operatorname{diag}(a)O^\top$, $Q=O\operatorname{diag}(q)O^\top$, and
\[
b(x,u)=Ax+Ou,\qquad \sigma=0.5I_d,\qquad
f(x,u)=-\tfrac12x^\top Qx-\sum_{i=1}^{16}(u_i^2-1)^2.
\]
The construction draws $a_i\in[-1.5,-0.5]$ and $q_i\in[0.5,1.5]$ independently and uses $(\lambda,\beta)=(0.3,0.5)$. The learner receives $A,Q$ and calls to the drift $b(x,u)$, but not $O$ or the rotated state $O^\top x$. The rotation is used only to compute the reference.

In reference coordinates $y=O^\top x$, the sixteen scalar HJB equations separate:
\begin{equation}\label{eq:1d-hjb}
\begin{aligned}
\beta w_i(y)&=-\tfrac12q_i y^2+a_i y w_i'(y)+\tfrac18w_i''(y)+F(w_i'(y)),\\
F(z)&=\lambda\ln\int_{\mathbb R}
\exp\!\left(\frac{-(u^2-1)^2+zu}{\lambda}\right)du.
\end{aligned}
\end{equation}
We numerically solve each equation on $1201$ nodes in $[-6,6]$. Thus $v_{\rm ref}(x)=\sum_iw_i((O^\top x)_i)$ and the reference policy is the product of the scalar Gibbs densities associated with $z_{{\rm ref},i}(x)=w_i'((O^\top x)_i)$.

The critic is a quadratic plus $96$ trainable tanh units,
\[
v_\theta(x)=-\tfrac12x^\top\Pi_\theta x-p_\theta^\top x-c_\theta
+\sum_{j=1}^{96}a_j\tanh(w_j^\top x+b_j),
\]
with $1881$ parameters. The actor transports independent Gaussian noise coordinatewise:
\[
U_i=T_{\phi,i}(x,\xi_i)=m_i(x)+e^{s_i(x)}\xi_i+
\sum_{k=1}^{3}e^{\alpha_{ik}(x)}
\tanh\!\left(\frac{\xi_i-\ell_{ik}(x)}{0.35}\right),
\qquad \xi_i\overset{\rm iid}{\sim}\mathcal N(0,1).
\]
Each of its eight coefficient functions per coordinate is linear in
\[
\bigl(1,\tanh(x_1/2.5),\ldots,\tanh(x_{16}/2.5)\bigr),
\]
giving $2176$ actor parameters. It contains no built-in double-well density. Monotonicity in $\xi_i$ supplies the actor density and score, and Gaussian-noise quadrature supplies the actor and critic expectations.

For the actor update, the required critic value-gradient component is obtained from the known drift through
\[
z_{v_\theta,i}(x)=\nabla_xv_\theta(x)\cdot\bigl[b(x,e_i)-b(x,0)\bigr],
\]
which is exactly the $i$th component of $B^\top\nabla_xv_\theta$ for an affine drift and the training code never reads $O$ explicitly. The critic starts at the exact value of the initial actor, which is not the optimal value. Then actor and critic take $2000$ coupled, actor-first updates with $\Delta\tau=0.1$, with actor damping $0.1$ and critic damping $0.01$. Each step uses $256$ states from a $1024$-state pool simulated under the current actor and refreshed every $50$ updates. 

\begin{table}[!htbp]
\centering\small
\begin{tabular}{lr}
Held-out quantity & Error\\\hline
Critic value gradient $B^\top\nabla v_\theta$, relative RMS & $8.57\%$\\
Mean control, RMS per coordinate & $0.0360$\\
Mean marginal TV, initial actor & $0.597$\\
Mean marginal TV, trained actor & $0.128$\\
Largest observed marginal TV, trained actor & $0.304$
\end{tabular}
\caption{The $2000$-update run evaluated on $200$ independent states from the learned actor's occupation law. TV compares corresponding one-dimensional marginals, not the joint policy. The value comparison includes actor suboptimality.}\label{tab:doublewell16-results}
\end{table}

Figure~\ref{fig:doublewell16-marginals} compares all sixteen marginals at the chosen nonzero state $x_0=O(0.8\cdot\mathbf1_{16})$. The trained actor learns two wells and their unequal masses without reference samples or rotated coordinates in training, but retains excess density between the wells. On the $200$ held-out occupation states, $v_\theta-v_{\rm ref}$ has mean $-0.828$ and centered RMS $0.047$. This level shift dominates the $37.9\%$ relative RMS value error. The mean joint Gibbs gap $\lambda\operatorname{KL}(\pi_\phi(\cdot\mid x)\Vert\pi_{\rm ref}(\cdot\mid x))$ is $0.401$, while the critic's mean policy-evaluation residual is $-0.032$. These diagnostics indicate that actor suboptimality accounts for most of the value gap.

\begin{figure}[!htbp]
\centering\includegraphics[width=\linewidth]{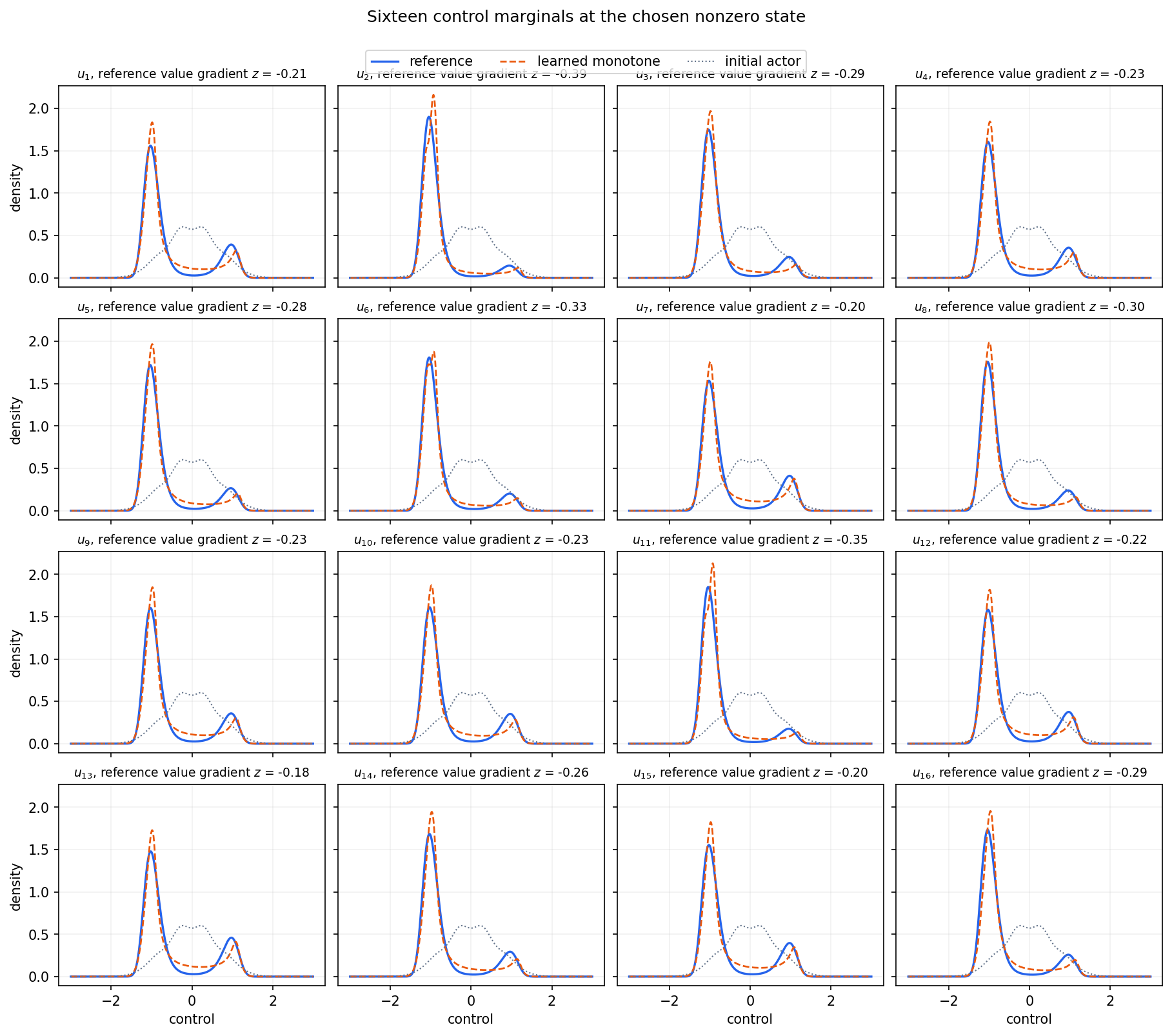}
\caption{Sixteen control marginals at $x_0=O(0.8\,\mathbf1)$: numerical reference, initial actor and learned monotone actor. }\label{fig:doublewell16-marginals}
\end{figure}

{\bibliographystyle{agsm}
\bibliography{ref}}

\end{document}